\RequirePackage{fix-cm}
\documentclass[smallextended]{svjour3}       % onecolumn (second format)
\smartqed  % flush right qed marks, e.g. at end of proof
\usepackage{graphicx}
\begin{document}

\title{Set-valued functions of bounded generalized variation and set-valued Young integrals%\thanks{Grants or other notes
}
%\subtitle{Set-valued $BV_p$ functions  and set-valued Young integrals}

\titlerunning{Set-valued $BV_p$ functions  and set-valued Young integrals}        % if too long for running head

\author{Mariusz Michta         \and
        Jerzy Motyl %etc.
}


\institute{Mariusz Michta \at
              Faculty of Mathematics, Computer Science and Econometrics\\
University of Zielona G\'{o}ra\\
Szafrana 4a\\
65-516\ Zielona G\'{o}ra\\
Poland\\
              \email{m.michta@wmie.uz.zgora.pl}           %  \\
           \and
           Jerzy Motyl \at
              Faculty of Mathematics, Computer Science and Econometrics\\
University of Zielona G\'{o}ra\\
Szafrana 4a\\
65-516\ Zielona G\'{o}ra\\
Poland\\
\email{j.motyl@wmie.uz.zgora.pl (corresp. author)}
}

\date{Received: date / Accepted: date}
\maketitle

\begin{abstract}
The paper deals with some properties of set-valued functions having a bounded  Riesz p-variation. Set-valued integrals of a Young type for such multifunctions are introduced. Selection results and properties of such set-valued integrals are discussed. These integrals contain as a particular case set-valued stochastic integrals with respect to a fractional Brownian motion, and therefore, their properties are crucial for the investigation of solutions to stochastic differential inclusions driven by a fractional Brownian motion.
\end{abstract}

%%% ----------------------------------------------------------------------
\maketitle

\section{Introduction}
\label{intro}

Since the pioneering work of R.J. Aumann in 1965 \cite{Aumann}, the notion of set-valued integrals for multivalued functions has attracted the interest of many authors both from theoretical and practical points of view. In particular, the theory has been developed extensively, among others, with applications to optimal control theory, mathematical economics, theory of differential inclusions and set-valued differential equations, see e.g., \cite{Ahmed1}, \cite{Aubin}, \cite{AubinCell}, \cite{Kis0}, \cite{Laksh5}, \cite{Tolst}. Later, the notion of the integral for set-valued functions has been extended to a stochastic case, where set-valued It\^o integrals have been studied . Moreover, concepts of set-valued integrals, both deterministic and stochastic, were used to define the notion of fuzzy integrals applied in the theory of fuzzy differential equations, e.g., \cite{DiamKloeden}, \cite{LakshTolst}.  On the other hand, in a single-valued case, one can consider integration with respect to integrators such as fractional Brownian motion  which has H\"{o}lder continuous sample paths. In some cases such integrals can be understood in the sense of Young (\cite{Young}). Controlled differential equations driven by Young integrals have been studied by A. Lejay in \cite{Lelay}. A more advanced approach to controlled differential equations is based on the rough path integration theory initiated by T. Lions (\cite{Lyons}) and further examined in \cite{Coutin}, \cite{Friz}. Control and optimal control problems inspired the intensive expansion of differential and stochastic set-valued inclusions theory. Thus it seems reasonable to investigate also differential inclusions driven by a fractional Brownian motion and Young type integrals also. Recently, in \cite{Bailleul} the authors considered a Young type differential inclusion, where solutions were understood as Young integrals of appropriately regular selections of multivalued right-hand side. Set-valued Aumann or It\^o type integrals are useful toolls in the investigation of properties of solution sets to differential or stochastic inclusions and set-valued equations \cite{Ahmed2}, \cite{Fei}, \cite{Gorn}, \cite{Kis1}. Therefore, it is quite natural to introduce set-valued Young type integrals.  Motivated by this, the aim of this work is to introduce such set-valued  integrals and to investigate their properties, especially these which seem to be useful in the Young set-valued inclusions theory. It is known, that three of properties of Aumann set-valued integrals are crucial in the differential inclusions theory. Namely they are, the existence of a Castaing representation of the set of integrable selectors,  decomposability of this set and valuation of a Hausdorff distance between set-valued integrals by the distance between integrated multifunctions (see e.g., \cite{Hiai}).

Set-valued Young integrals considered in the paper deal with the class of set-valued functions having a bounded Riesz p-variation. Such integrals contain as a particular case set-valued stochastic integrals with respect to a fractional Brownian motion. Therefore, in our opinion, their properties are crucial not only for the existence of solutions to stochastic differential inclusions and set-valued stochastic differential equations driven by a fractional Brownian motion but also for useful properties of their solution sets.

\bigskip

The paper is organized as follows. In Section 2, we define a space of set-valued functions of a finite Riesz $p$-variation. Section 3 deals with the properties of sets of appropriately regular selections of such set-valued functions. Here we shall establish a new type of decomposability for sets of functions with a finite Riesz $p$-variation as well as their integral property. Finally, in Section 4, we introduce a set-valued Young type integral which is based on the sets of selections examined in Section 3. We shall investigate properties of this set-valued integral. 

\section{Finite p-variation set-valued functions }

Let $(X,\|\cdot \|)$ be a Banach space.  Denote by $Comp\left( X\right) $ and $Conv\left( X\right) $ the families of all nonempty and compact, and nonempty compact and convex subsets of $X$, respectively. The Hausdorff metric $H_{X}$ in $Comp\left( X\right) $ is defined by
$$
H_{X}\left( B,C\right) =\max \left\{ \overline{H}_{X }\left( B,C\right) ,\overline{H}_{X}\left( C,B\right) \right\} ,
$$
where $\overline{H}_{X}\left( B,C\right) =\sup_{b\in B}\mathrm{dist}_{X}\left( b,C\right) =\sup_{b\in B}\inf_{c\in C}\|c-b\|_X$. If $X$ is separable, then the space $\left( Comp\left( X\right) ,H_{X}\right) $ is a Polish  space and $\left( Conv\left( X\right),H_{X}\right) $ is its closed subspace. For $B,C,D,E\in Comp\left(X\right) $ we have, 
\begin{equation}
H_{X}\left( B+ C,D+ E\right) \leq H_{X}\left(B,D\right) +H_{X}\left( C,E\right)  \label{32}
\end{equation}
where $B+ C:=\left\{ b+c:b\in B,c\in C\right\} $ denotes the Minkowski sum of $B$ and $C$. Moreover, for $B,C,D\in Conv\left( X\right) $ the equality 
\begin{equation}
H_{X}\left( B+ D,C+ D\right) =H_{X}\left(B,C\right) \hbox{,}  \label{33}
\end{equation}
holds, see e.g., \cite{Laksh5} for details.

We use the notation
$$
\left\Vert A\right\Vert _{X}:=H_{X}\left( A,\left\{0\right\} \right) =\sup_{a\in A}\left\Vert a\right\Vert _{X}\hbox{ for }A\in Conv\left( X\right).
$$

\noindent Let $T>0$ and $\beta \in (0,1]$. For every function $f:[0,T]\rightarrow X$ we define
\[\left\Vert f\right\Vert _{\infty }=\sup_{t\in \lbrack 0,T]}\left\Vert f(t)\right\Vert _{X} \hbox{ and } M_{\beta }(f)=\sup_{0\leq s<t\leq T}\frac{\|f(t)-f(s)\|_{X}}{\left( t-s\right) ^{\beta }}.\]
By $\mathcal{C}^{\beta }\left( X\right) $ we denote the space of $\beta $-H\"{o}lder-continuous ( or shortly $\beta $-H\"{o}lder) functions  with a finite norm
$$
\left\Vert f\right\Vert _{\beta }:=\left\Vert f\right\Vert _{\infty
}+M_{\beta }(f).
$$

It can be shown that $\left( \mathcal{C}^{\beta },\| \cdot \|_{\beta }\right) $ is a Banach space. Similarly, for a set-valued function $F:[0,T]\rightarrow Comp\left( X\right) $ let
$$
\left\Vert F\right\Vert _{\beta }:=\left\Vert F\right\Vert _{\infty}+M_{\beta }(F)
$$
where
\[\| F\| _{\infty }=\sup_{t\in \lbrack 0,T]}\| F( t) \|_{X} \hbox{ and }M_{\beta }(F)=\sup_{0\leq s<t\leq T}\frac{H_{X}( F(t),F(s)) }{( t-s)^{\beta }}.\]
 A set-valued function $F$ is said to be $\beta $-H\"{o}lder if $\left\Vert F\right\Vert _{\beta }<\infty $. By $\mathcal{C}^{\beta }(Comp( X))$ we denote the space of all such set-valued functions. The space of $\beta $-H\"{o}lder set-valued functions having compact and convex values will be denoted by $\mathcal{C}^{\beta }(Conv(X))$.

\bigskip 

Let $(E,d)$ be a metric space. For every $0\leq a <b \leq T$, by $\Pi _n=\{t_i\}_{i=0}^n$ we denote a partition $a=t_0<t_2<...<t_n=b$ of the interval $[a,b]$. For every function $f:[0,T]\rightarrow E$ and $1\leq p< \infty $ we define its Young p-variation on $[a,b]$ by the formula
\[Var_p(f,[a,b])= \sup _{\Pi} \sum _{i=1}^{n}\big(d(f(t_{i-1}),f(t_i)\big)^p\] 
and a Riesz p-variation on $[a,b]$ by the formula
\[V_p(f,[a,b])= \sup _{\Pi} \sum _{i=1}^{n}\frac{\big(d(f(t_{i-1}),f(t_i)\big)^p}{(t_{i}-t_{i-1})^{p-1}}.\] 
We denote $Var_p(f,[0,T])$ by $Var_p(f)$ and $V_p(f,[0,T])$ by $V_p(f)$, respectively. If $Var_p(f)<\infty$ (resp., $V_p(f)<\infty)$ we call $f$ a bounded Young (resp., Riesz) p-variation function. The class of all functions of bounded p-variations will be denoted by $BVar_p([0,T],E)$ or $BV_p([0,T],E)$, respectively. In the sequel we denote spaces $BVar_p([0,T],E)$ and $BV_p([0,T],E)$ simply by $BVar_p(E)$ and $BV_p(E)$, respectively. If $(X,\|\cdot\|_{X})$ is a Banach space, then $BVar_p(X)$ or $BV_p(X)$  with norms $ \|f\|_{Var_p}=\sup_{t\in[0,T]}\|f(t)\|_X\! +\!(Var_p(f))^{1/p}$ and $ \|f\|_{V_p}=\sup_{t\in[0,T]}\| f(t)\|_X+(V_p(f))^{1/p}$, respectively, are Banach spaces. For $X=R^d$ and considered with the Euclidean norm we will use the notation $\|x\|$ instead of $\|x\|_{R^d}$.

\bigskip
We collect some properties of functions of bounded $V_p$-variation in the following proposition.

\bigskip

\begin{proposition}
\label{Prop1} (\cite{ChistGalkin}, \cite{Chi}). Let $f:[0,T]\rightarrow E$. Then, for every $1\leq p< \infty $, the following conditions hold:
\smallskip

(a) For every $[a,b]\subset [0,T]$ and  $ a\leq t\leq b$ we have
\[V_p(f,[a,t])+V_p(f,[t,b])= V_p(f,[a,b]).\]
\smallskip

(b) if $f\in BV_p(E)$, then $V_1(f,[a,b])\leq (b-a)^{1-1/p}\big(V_p(f,[a,b])\big)^{1/p}$ for every $[a,b]\subset [0,T]$, (Jensen inequality).
\smallskip

(c) if $(f_n)$ is a sequence such that $\lim_{n\rightarrow\infty} d(f_n(t),f(t))=0$ for every $t\in [a,b]$ then $V_p(f,[a,b])\leq \liminf_{n\rightarrow\infty}V_p(f_n,[a,b]).$
\smallskip
 
(d) if $X$ is a reflexive Banach space and $f\in BV_p(X)$, then $f$ admits a strong derivative $f'$ and $V_p(f,[a,b])=\int _a^b\|f'(t)\|_X^pdt$, (Riesz theorem). 
\end{proposition}

Let $(X,\|\cdot \|)$ be a Banach space and let $\Pi _m :0=t_{0}<t_{1}<...<t_{m}=T$ be a partition of the interval $[0,T]$. Given a set-valued function $F:[0,T]\rightarrow Comp\left( X\right) $ we set
$$
V_{p}(F,\Pi _m ):=\sum\limits_{i=1}^{m}\frac{H_{X}^{p}(F\left(t_{i}\right) ,F\left( t_{i-1}\right) )}{(t_{i}-t_{i-1})^{p-1}}.
$$
Then by a Riesz p-variation on $[0,T]$ we mean the quantity 
$$
V_{p}(F):=\sup_{\Pi _m }V_{p}(F,\Pi _m ).
$$

By  $BV_{p}(Comp\left( X\right) )$ we denote the space of all set-valued functions from $[0,T]$ into $Comp\left( X\right) $ having finite  Riesz $p$-variation. 

\section{Selections of finite p-variation set-valued functions}

Let $T>0$ be given and let $F:[0,T]\rightarrow {\rm Comp}(X)$ be a measurable set-valued function. A measurable function $f:[0,T]\rightarrow X$ is called a measurable selection of $F$ if $f(t)\in F(t)$ for all $t\in [0,T]$. For $1\leq p< \infty $, define the set
\[S_{L^{p}}(F)=\{f\in L^p([0,T],X):\;f(t)\in F(t) {\;\rm a.e.\;t\in [0,T]}\}.\]
$S_{L^{p}}(F)$ is a closed subset of $L^p([0,T],X)$. It is nonempty if $F$ is p-integrably bounded i.e., if there exists $g\in L^p([0,T]$ such that $\|F(t)\|_X\leq g(t)$ for a.e. $t\in [0,T]$. In such a case there exists a sequence $(f_n)\subset S_{L^{p}}(F)$ such that $F(t)=\overline{\{f_n(t)\}_{n=1}^{\infty}}$ for all $t\in [0,T]$. The sequence $(f_n)$ is called an $L^p$-Castaing representation for $F$. For other properties of measurable set-valued functions and their measurable selections see e.g., \cite{AubinFr}.

\begin{definition}
\label{Def1} {\rm Let  $F:[0,T]\rightarrow {\rm Comp}(R^d)$ be a set-valued function. For $1\leq p< \infty $, define 
%\[S_{Var_{p}}(F):=\{f\in BVar_{p}(R^{d}):f(t)\in F(t),t\in\lbrack 0,T]\}  \label{selVar}\]
%and
\[S_{V_{p}}(F):=\{f\in BV_{p}(R^{d}):f(t)\in F(t),t\in \lbrack 0,T]\},  \label{selV}\]
the set of selections of $F$ with a bounded 
%Young p-variation and bounded 
Riesz p-variation.}
%, respectively
\end{definition}

Let $F\in \mathcal{C}^{\beta }\left( Comp\left( R^{d}\right)\right) $. Such set-valued functions need not admit any H\"{o}lder or even continuous selection, see e.g., \cite{ChistGalkin}. However, considering the smaller class $BV_p\left( Comp\left( R^{d}\right)\right)\subset \mathcal{C}^{\beta }\left( Comp\left( R^{d}\right)\right)$, the following selection theorem holds true. 

\begin{proposition}
\label{prop2} (\cite{Chi}) Let  $F:[0,T]\rightarrow {\rm Comp}(R^d)$ be a set-valued function. If $F\in BV_p({\rm Comp}(R^d))$ for some $1\leq p< \infty $ then there exist a function $\phi \in BV_p(R^d)$ and a sequence of equi-Lipschitzean functions $(g_n)_{n=1}^{\infty}$ with Lipschitz constants $L_n\leq 1$ such that taking $f_n:=g_n\circ\phi$, we have $V_p(f_n,[a,b])\leq V_p(F,[a,b])$ for every $0\leq a<b\leq T$ and $F(t)=\overline{\{f_n(t)\}_{n=1}^{\infty}}$ for every $t\in [0,T]$. The set $\{f_n \}_{n=1}^{\infty}$ is a $V_p$-Castaing representation  for $F$.\label{Prop9a} 
\end{proposition}

Let us note, that the set $S_{V_{p}}(F)$ need not be closed in the topology of point convergence even if $F$ is bounded. 
\begin{example}
\label{Ex1} {\rm The set $S_{V_{p}}(F)$ need not be closed in the topology of point convergence even if $F$ is bounded. To see this, let $W$ be a Wiener process defined on some adequate probability space $(\Omega, \mathcal{F},P)$. Let $W(\cdot , \bar{ \omega} )$ denote its trajectory connected with a fixed $\bar{ \omega}\in \Omega$. Then $M=\sup_{t\in [0,T]}| W(t , \bar{ \omega} )|<\infty $, because of continuity of trajectories of a Wiener process. Let $F:[0,T]\rightarrow Comp(R^1)$ be a set-valued function defined by formula $F(t)=[-M,M]$ for every $t\in [0,T]$. Let  $(\Pi _n)_{n=1}^{\infty}=(\{t_i\}_{i=1}^n)_{n=1}^{\infty}$ denote a sequence of normal partitions $0=t_1<t_2<...<t_n=T$ of the interval $[0,T]$ and let $W_n(\cdot , \bar{ \omega} )$ denote regularizations of $W(\cdot , \bar{ \omega} )$ defined by the formula below 
\[\begin{array}{l} \;W_n(t , \bar{ \omega} )= \left\{ \begin{array}{cl} W(t_i , \bar{ \omega} ) &{\rm ~ for}
\;\;\;t=t_i \\
\hbox{is linear} & {\rm ~ for}\;\;\;t\in(t_i,t_{i+1}) \\
\end{array}
\right. 
\end{array}.
\]
It is clear that $W_n(t , \bar{ \omega})\in F(t)$. Moreover, for a linear function $g(t)=at+b$, we have $V_p(g,[t_i,t_{i+1}])=|a|^p(t_{i+1} -t_i)<\infty$. Therefore, we get by Proposition \ref{Prop1}(a),
\[V_p(W_n(\cdot , \bar{ \omega}),[0,T] )=\sum^{n-1}_{i=1}V_p(W_n(\cdot , \bar{ \omega}),[t_i,t_{i+1}])\]
\[\leq \max \{\frac{|W(t_{i+1}, \bar{ \omega})-W(t_i , \bar{ \omega})|^p}{(t_{i+1}-t_i)^p}, i=1,2,...,n-1\}\cdot \sum^{n-1}_{i=1}(t_{i+1}-t_i)<\infty .\]
It means that $W_n(\cdot , \bar{ \omega})\in S_{V_{p}}(F)$. But $W_n(t, \bar{ \omega})$ tends to $W(t, \bar{ \omega})$ for every $t\in [0,T]$. Since $V_p(W(\cdot, \bar{ \omega})=+\infty$ for every $1\leq p<2$, then $W(\cdot, \bar{ \omega})\notin S_{V_{p}}(F)$.}
\end{example}

However,  the set $S_{V_{p}}(F)$  is closed in the norm $ \|\cdot\|_{V_p}$ because of Jensen inequality $\|f_n(t)-f(t)\|\leq \max\{1,T^{1-1/p}\}\|f_n-f\|_{V_p}\rightarrow 0$ and Proposition \ref{Prop1}(c).

\begin{proposition}  
 Let  $F:[0,T]\rightarrow {\rm Comp}(R^d)$ be a set-valued function, $F\in BV_p({\rm Comp}(R^d))$ for some $1\leq p< \infty $. Let $\{f_m \}_{m=1}^{\infty}$ be the $V_p$-Castaing representation of $F$ given in Proposition \ref{prop2}. Then, for every $f\in S_{V_{p}}(F)$ and every $\epsilon >0$, there exist a finite measuarable covering $A_1,...,A_n$ of the interval $[0,T]$ and functions $f_{k_1},...,f_{k_n}\in \{f_m \}_{m=1}^{\infty}$ such that 
\[\|f-\sum_{j=1}^n{\rm 1}\hspace{-1mm}{\rm I} _{A_j}\cdot f_{k_j}\|_{L^p}<\epsilon.\]
Moreover, for every $f\in S_{V_{p}}(F)$ and every $\epsilon >0$, there exist $n\geq 1$, a partition $\Pi _n:0=t_0<t_1<...<t_n=T$ and functions $f_{k_0},...,f_{k_n}\in \{f_m \}_{m=1}^{\infty}$ such that 
\[\|f-\sum_{j=0}^{n-1}{\rm 1}\hspace{-1mm}{\rm I} _{[t_j,t_{j+1})}\cdot f_{k_j}\|_{\infty}<\epsilon.\]
\label{Prop3} 
\end{proposition}
\begin{proof} {\rm  Since $S_{V_{p}}(F)\subset S_{L^p}(F)$ and the $V_p$-Castaing representation of $F$ is also an $L^p$-Castaing representation of $F$ introduced in \cite{Cas1}, then the proof follows by Lemma 1.3 of \cite{Hiai}.}

{\rm We prove second inequality. Let $f\in S_{V_{p}}(F)$ be arbitrary taken. There exists $\delta $ such that $\|f(t)-f(s)\|< \epsilon/3$ and $\|f_m(t)-f_m(s)\|< \epsilon/3$ for every $|t-s|<\delta $ (see Proposition \ref{Prop1}). Let us take a partition $\Pi _n:0<\delta <2\delta <...<n\delta<T$. Since $f(t)\in  \overline{\{f_m(t)\}_{m=1}^{\infty}}$, then for every $k=0,1,...,n$ there exists $m_k$ such that $\|f(k\delta )-f_{m_k}(k\delta )\|<\epsilon/3$. Therefore, $\|f(t)-f_{m_k}(t)\|< \epsilon$ for $t\in[k\delta,\min\{(k+1)\delta ,T\}]. $

Thus,
\[ \|f-\sum_{k=0}^{n-1}{\rm 1}\hspace{-1mm}{\rm I} _{[t_k,t_{k+1})}\cdot f_{m_k}\|_{\infty}=\|f - \sum _{k=0}^{n-1}{\rm 1}\hspace{-1mm}{\rm I} _{[k\delta,\;\min\{(k+1)\delta ,\;T\}]}\cdot f_{m_k}\|_{\infty}\leq\epsilon .\]}
\end{proof}

\bigskip

Let us note, that a similar approximation property with respect to $V_p$-variation norm need not hold true.

\bigskip

Now we introduce the notion of $V_p$-decomposable selections of set-valued functions and investigate their properties.

\bigskip

Let $(\Omega,\mathcal{A},\mu)$ be a measure space. A set $\Lambda\subset L^p(\Omega,\mathcal{A},{\rm I}\!{\rm R}^d)$  is said to be $L^p$-decomposable, if for every $f_1,f_2\in \Lambda$ and every $A\in\mathcal{A}$ one has ${\rm 1}\hspace{-1mm}{\rm I}_A\cdot f_1+{\rm 1}\hspace{-1mm}{\rm I}_{A^\sim}\cdot f_2\in \Lambda$, where $A^\sim$ denotes the complement of the set $A$ in $\Omega$. For any $L^p$-decomposable sets $\mathcal{H},\mathcal{K}\subset L^p(\Omega,\mathcal{A},{\rm I}\!{\rm R}^d)$, the Minkowski sum $\mathcal{H}+\mathcal{K}$ is again an $L^p$-decomposable subset of the space $L^p(\Omega,\mathcal{A},{\rm I}\!{\rm R}^d)$. 

For a given set $B\subset L^p(\Omega,\mathcal{A},{\rm I}\!{\rm R}^d)$ we denote the set $\{\sum_{k=1}^{n}{\rm 1}\!\!\hspace{-1mm}{\rm I}_{A_k}\cdot \beta_k:\;A_k\in\mathcal{A},\;\beta_k\in B,\;n=1,2,...\}$  by ${\rm dec}_{L^p}(B)$ and call it an $L^p$-decomposable  hull of a set $B$. 

By $\overline{{\rm dec}}_{L^p}(B)$ we denote a closed $L^p$-decomposable hull of a set $B$. Similarly as in the case of convex and closed convex hulls, they are  the smallest $L^p$-decomposable and closed $L^p$-decomposable sets containing the set $B$, respectively. 

From this it follows that the set $S_{L^p}(F)$ consisting of all $L^p$-selectors of a given measurable set-valued function $F$ is always $L^p$-decomposable and therefore, $S_{L^p}(F)=\overline{{\rm dec}}_{L^p}(S_{L^p}(F))$. Conversely, if a closed set $\Lambda\subset L^p(\Omega,\mathcal{A},{\rm I}\!{\rm R}^d)$ is $L^p$-decomposable, then there exists a measurable set-valued function $F:\Omega\rightarrow R^d$ such that $\Lambda = S_{L^p}(F)$, (see \cite{Hiai}). For other properties of $L^p$-decomposable sets, see \cite{Frysz}.

$L^1$-decomposability of the set of $L^1$-selectors of a given measurable set-valued function $F$ is crucial for investigating properties of a set-valued Aumann integral of $F$defined by the formula
\[\int _A F(t)\;d\mu=\{\int _Af(t)\;d\mu:\;f\in S_{L^1}(F)\}.\]
Unfortunately, the set $S_{V_p}(F)$ need not be $L^p$-decomposable for any $p\geq1$ and therefore, if one defines a set-valued Young integral in the Aumann's sense, it is difficult to obtain its reasonable properties. This leads to the idea of a different type of decomposability called $V_p$-decomposability.

\bigskip

It follows from Proposition \ref{Prop1}(d) that a function $f$ belongs to $ BV_{p}(R^{d})$ if and only if its strong derivative $f'$ belongs to $L^p([0,T])$, $f(t)=f(0)+\int_0^tf'(s)ds$ and $V_p(f,[0,t])=\int _0^t\|f'(s)\|^pds$ for every $t\in [0,T]$. This property has been inspiring to the following definition.

\begin{definition}
\label{Def2} {\rm A set $\Lambda \subset BV_p\left( R^{d}\right)$ is {\it $V_p$-decomposable} (decomposable in the sense of its Riesz $p$-variation) if for every $f_1,f_2\in \Lambda$ and every $a\in [0,T]$ the function $f=f_1\oplus_a f_2$ defined by 
\[f(t)=f_1(0)+\int _0^t \big({\rm 1}\hspace{-1mm}{\rm I}_{[0,a)}(s)\cdot f_1'(s)+{\rm 1}\hspace{-1mm}{\rm I}_{[a,T]}(s)\cdot f_2'(s)\big)ds\]
belongs to the set $\Lambda$.

For a given set $B\subset BV_p\left( R^{d}\right)$ by ${\rm dec}_{V_p}(B)$ we denote a $V_p$-decomposable hull of a set $B$, i.e., the smallest $V_p$-decomposable  set containing the set $B$. }
\end{definition}

\begin{remark}
\label{Rem1} {\rm Every function $f=f_1\oplus_a f_2$ from Definition \ref{Def2} can be represented by the formula
 
\[\begin{array}{l} \;f(t)= \left\{ \begin{array}{cl} f_1(t) &{\rm ~ for}
\;\;\;0\leq t< a \\
f_2(t)-f_2(a)+f_1(a) & {\rm ~ for}\;\;\;a\leq t\leq T \\
\end{array}
\right. 
\end{array}.
\]
Moreover, for every $B\subset BV_p(\left( R^{d}\right)$ we have
\[{\rm dec}_{V_p}(B)=\{f\in BV_p\left( R^{d}\right):\;f(t)=f_1(0)+\int_0^t\big(\sum_{i=0}^{m-1}{\rm 1}\hspace{-1mm}{\rm I}_{[t_i,t_{i+1})}(s)\cdot f_i'(s)\big)ds:\]
$\Pi_m:\;0=t_0<...<t_m=T,\; m=1,2,...,\;\;f_i\in B,\;i=1,...,m\}.$}
\end{remark}

\begin{definition}
\label{Def3} {\rm A set $\mathcal{R}\subset BV_p\left( R^{d}\right)$ is called {\it an integral} if there exist $x_0\in R^d$ and a measurable and $p$-integrably bounded set-valued function $\Phi :[0,T]\rightarrow Conv(R^d)$ such that }
$$
\mathcal{R} = x_0+\!\int \Phi(s)ds=\! \{f\in BV_p(R^d)\!: f(\cdot )=x_0+\!\int _0^{\cdot }\!\phi(s)ds,\;\phi\in S_{L^p}(\Phi)\}.
\label{cra}
$$
We denote by $\mathcal{R}(t)$ the set 
\[\mathcal{R} (t)=\{f(t): f(\cdot )\in \mathcal{R}\}=\{x_0+\int_0^t \phi(s)ds,\;\phi\in S_{L^p}(\Phi)\}\]. 
\end{definition}

\begin{theorem}
\label{Th1}  Let $\mathcal{R}\subset BV_p(R^d)$ be an integral. Then $\mathcal{R} $ is closed with respect to the norm $\|\cdot \|_{\infty}$ and $V_p$-decomposable.
\end{theorem}

\begin{proof} {\rm If $\mathcal{R}$ is an integral then for every $t\in [0,T]$ $\mathcal{R} (t)$ is a closed subset of $R^d$ by Theorem 8.6.7 of \cite{AubinFr}.
 Let $(f_n)_{n=1}^{\infty}\subset \mathcal{R} $ be a sequence convergent to some $f$ with respect to the norm $\|\cdot \|_{\infty}$. Since $\mathcal{R} $ is an integral, then $f_n(t)=x_0+\int _0^t\phi _n(s)ds$ for some $\phi _n\in S_{L^p}(\Phi)$. But $f_n(0)=x_0$ and therefore, $f(0)=x_0$.  Moreover, since $\Phi $ is $p$-integrably bounded by some function $g\in L^p([0,T])$ then $\sup_n V_p ( f_n ) \leq \|g\|_{L^p}$. It follows from Proposition \ref{Prop1}(c) that $V_p(f)\leq \|g\|_{L^p}$. Therefore, $f\in BV_p(R^d)$ and $f(t)=x_0+\int _0^t f'(s)ds$. Since $\Phi$ is $p$-integrably bounded and has closed and bounded values, then the set $S_{L^p}(\Phi)$ is closed, bounded and convex in $L^p([0,T])$. Therefore, it is weakly compact there. Thus there exists a subsequence $(\phi_{n_k})$ of $(\phi_n)$ weakly convergent to some $\phi\in S_{L^p}(\Phi)$. Let $J:L^p([0,T])\rightarrow C([0,T])$ be a linear operator defined by formula $J(\psi)=x_0+\int_0^{\cdot }\psi(s)ds$. Since $J$ is norm-to-norm continuous, then it is also weak-to-weak continuous. Thus $f_{n_k}=x_0+\int_0^{\cdot }\phi_{n_k}(s)ds$ tends weakly to  $x_0+\int_0^{\cdot }\phi(s)ds$. But $(f_{n_k})$ tends to $f=x_0+\int_0^{\cdot }f'(s)ds$ in $\|\cdot \|_{\infty}$ norm. Thus $\phi=f'$ and therefore, $f'\in S_{L^p}(\Phi)$. This implies $f\in \mathcal{R}$, which proves the closedness of $\mathcal{R}$.

Now let us take $f_1,f_2\in \mathcal{R}$. There exist a set-valued function $\Phi$ and functions $\phi _1, \phi_2 \in S_{L^p}(\Phi)$ such that $f_1(t)=x_0+\int _0^t\phi _1(s)ds$ and  $f_2(t)=x_0+\int _0^t\phi _2(s)ds$ for every $t\in [0,T]$. Let $a\in [0,T]$ be arbitrarily taken and let $\gamma (t)= {\rm 1}\hspace{-1mm}{\rm I}_{[0,a)}(s)\cdot \phi_1(s)+{\rm 1}\hspace{-1mm}{\rm I}_{[a,T]}(s)\cdot \phi_2(s)$. Then $\gamma \in S_{L^p}(\Phi)$ and therefore, $f=f_1\oplus_a f_2=x_0+\int \gamma (s)ds\in \mathcal{R}$. It means that $\mathcal{R} $ is $V_p$-decomposable.}
\end{proof}

\begin{theorem}
\label{Th2}  Let $\mathcal{R}\subset BV_p(R^d)$, $\mathcal{R} (0)=x_0$, be bounded, $V_p$-decomposable, convex and closed with respect to the norm $\|\cdot \|_{\infty}$. Then $\mathcal{R} $ is an integral. 
\end{theorem}

\begin{proof} {\rm Assume that $\mathcal{R}\subset BV_p(R^d)$, let $f_1,f_2\in \mathcal{R}$ and $a\in [0,T]$ be arbitrarily taken. If $f=f_1\oplus_a f_2$, then $f\in \mathcal{R}$ by the assumption of $V_p$-decomposability. We define the set $M$ by the formula
\[M=\{\phi \in L^p([0,T]):\; x_0+\int \phi (s)ds\in \mathcal{R}\}.\]
Then $M$ is convex in $L^p([0,T])$. It is bounded and closed in $L^p([0,T])$ by Proposition \ref{Prop1}(d).

Since $\mathcal{R} (0)=x_0$, then $\mathcal{R} =\{f_{\alpha}:\; f_{\alpha}=x_0+\int f_{\alpha}'(s)ds;\;f_{\alpha}'\in M\}$. We will show that the set $M$ is $L^p$-decomposable in $L^p([0,T],\beta([0,T]),\lambda )$, i.e., we will show that for every set $A\in \beta([0,T])$ and any $\phi, \psi \in M$, the function $\gamma={\rm 1}\hspace{-1mm}{\rm I}_{A}\cdot \phi+{\rm 1}\hspace{-1mm}{\rm I}_{ A^\sim}\cdot \psi$ belongs to the set $M$. $\beta ([0,T]) $, as usual, denotes here the Borel $\sigma$ algebra of subsets of the interval $[0,T]$ and $\lambda $ is a Lebesgue measure. 

We take a partition $\Pi _n: 0=t_0<t_1<...<t_{2n}<t_{2n+1}=T$ and the set $A$ of the form $A=\bigcup _{i=0}^{n}[t_{2i},t_{2i+1})$. Since $\mathcal{R} $ is $V_p$-decomposable it is easy to see that taking any $f_1,...,f_{2n+1}\in \mathcal{R}$ a function $f$ given by the formula $f(t)=x_0+\int _0^t\sum_{i=0}^{2n}{\rm 1}\hspace{-1mm}{\rm I} _{[t_i,t_{i+1})}(s)\cdot f_{i+1}'(s)ds$ belongs to $\mathcal{R}$. Therefore, taking $f_{2i}'=\phi$ for $i=1,2,...n$ and $f_{2i+1}'=\psi$ for $i=0,1,...n$, we have
\[x_0+\int \gamma (s)ds=x_0+\int \big({\rm 1}\hspace{-1mm}{\rm I}_{A}(s)\cdot \phi(s)+{\rm 1}\hspace{-1mm}{\rm I}_{ A^\sim}(s)\cdot \psi(s)\big)ds=x_0+\int_0^tf'(s)ds\in \mathcal{R}.\]
It means that $\gamma \in M$.

Let $\mathcal{M}=\bigcup _{n=1}^{\infty}\bigcup _{\Pi_n}\{B\subset [0,T]:\; B=\bigcup _{i=0}^{n-1}[t_{2i},t_{2i+1})\}$. Then $\mathcal{M} $ is a ring generating a $\sigma$-algebra $\beta ([0,T])$. We will show that $\mathcal{M}$ is a monotone class also. To this end, assume that  $(A_i)_{i=1}^{\infty} \subset \mathcal{M}$ and $A_i\subset A_{i+1}$. We prove that the set $A=\bigcup _{i=1}^{\infty}A_i$ belongs to $\mathcal{M}$. We can find an infinite partition $\Pi _{\infty}:0=t_0<t_1<t_2<...$ of $[0,T]$ and taking $\tilde{A_k}=\bigcup _{i=0}^k[t_{2i},t_{2i+1})$ we get $\tilde{A_k}\subset \tilde{A}_{k+1}\subset A$ and $A=\lim _{k\rightarrow \infty}\tilde{A_k}=\bigcup _{k=1}^{\infty}\tilde{A_k}$. Therefore, ${\rm 1}\hspace{-1mm}{\rm I} _{A}(s)=\lim _{k\rightarrow \infty}{\rm 1}\hspace{-1mm}{\rm I} _{\tilde{A_k}}(s)$ for every $s\in [0,T]$. Since the sets ${(\tilde{A_k})}^\sim$ form a decreasing family, then a sequence $({\rm 1}\hspace{-1mm}{\rm I} _{{(\tilde{A_k})}^\sim}(s))$ is a decreasing sequence of functions convergent to ${\rm 1}\hspace{-1mm}{\rm I} _{A^\sim }(s), $ where $A^\sim =\bigcap _{k=1}^{\infty}(\tilde{A_k})^\sim$.

Let us take any $\phi, \psi \in M$, $A=\bigcup _{i=0}^{\infty}[t_{2i},t_{2i+1})$ and let $\gamma (s)={\rm 1}\hspace{-1mm}{\rm I}_{A}(s)\cdot \phi(s)+{\rm 1}\hspace{-1mm}{\rm I}_{ A^\sim}(s)\cdot \psi(s)$. Then $\gamma (s)=\lim _{k\rightarrow \infty}\gamma _k(s)$, where $\gamma _k(s)={\rm 1}\hspace{-1mm}{\rm I}_{\tilde{A_k}}(s)\cdot \phi(s)+{\rm 1}\hspace{-1mm}{\rm I}_{{(\tilde{A_k})}^\sim}(s)\cdot \psi(s)$. It was shown in the first part of the proof that $\gamma _k(s)\in M$, because of $x_0+\int \gamma _k(s)ds\in \mathcal{R}.$ We show that $\gamma \in M$, i.e., that $f=x_0+\int \gamma (s)ds\in \mathcal{R}.$ We know that $f_k=x_0+\int \gamma _k(s)ds\in \mathcal{R}$. We have
\[\|f_k-f\|_{\infty}=\sup _{t\in[0,T]}\|\int _0^t (\gamma _k(s)-\gamma (s))ds\|\leq \int _0^T \|\gamma _k(s)-\gamma (s)\|ds.\]
However, $\gamma _k(s)\rightarrow\gamma (s)$ a.e. and the sequence $\|\gamma _k(s)-\gamma (s)\|$ admits a $p$-integrable majorant $2|\phi(s)|+2|\psi(s)|$. Therefore,  $\|f_k-f\|_{\infty}\rightarrow 0$. Since $\mathcal{R}$ is closed by the assumption, then $f\in \mathcal{R}$ and therefore, $\gamma\in M$.

We have shown that the set 
\[W=\{A\in \beta ([0,T]): {\rm 1}\hspace{-1mm}{\rm I}_{A}\cdot \phi+{\rm 1}\hspace{-1mm}{\rm I}_{ A^\sim}\cdot \psi \in M \hbox{ if }\; \phi, \psi \in M\}\]
contains a ring generating $\beta ([0,T])$ and a monotone class 

\[\Lambda = \bigcup _{\Pi _{\infty}}\{A, A^\sim \subset [0,T]:\;A=\bigcup _{i=0}^{\infty}[t_{2i},t_{2i+1}) \}.\] 
From the monotone class theorem we deduce that for every $\phi, \psi \in M$ and every set $Q\in \beta ([0,T])$ the set ${\rm 1}\hspace{-1mm}{\rm I} _Q\phi+{\rm 1}\hspace{-1mm}{\rm I} _{Q^\sim} \psi$ belongs to $M$. Therefore, $M$ is $L^p$-decomposable and by Theorem 3.1 of \cite{Hiai} there exists a measurable set-valued function $\Phi:[0,T]\rightarrow Cl(R^d)$ such that $S_{L^p}(\Phi)=M=\{\phi \in L^p([0,T]):\; x_0+\int \phi(s)ds\}\in \mathcal{R} $. It means that $\mathcal{R} =x_0+\int \Phi(s)ds$. Since $S_{L^p}(\Phi)=M$ is convex, then $\Phi$ has convex values by Theorem 1.5 from \cite{Hiai}.  Moreover, $\Phi$ is $p$-integrably bounded by the boundedness of $M$. Therefore, $\mathcal{R}$ should be an integral.}
\end{proof}

\begin{definition} Let $X$ be a real normed linear space. Let $A,B\in $ Conv$(X).$
The set $C\in $ Conv$(X)$ is said to be {\it the Hukuhara difference} $A\div B$ if $A=B+C.$ Consider a set-valued mapping $G:R^{1}\rightarrow $ Conv$(X).$ We say that $G$ admits {\it a Hukuhara differential} at $t_0\in R^{1}$, if there exists a set $D_{H}G(t_0)\in $ Conv$(X)$ and such that the limits
\[\lim_{\Delta t\rightarrow 0+}\frac{G(t_{0}+\Delta t)\div G(t_{0})}{\Delta t} \]
and
\[\lim_{\Delta t\rightarrow 0+}\frac{G(t_{0})\div G(t_{0}-\Delta t)}{\Delta t} \]
exist and are equal to the set $D_{H}G(t_{0})$. 
\end{definition}
For a detailed discussion of the properties and applications of the Hukuhara differentiable
multifunctions we refer the reader to \cite{Laksh5}.

\smallskip

Now we are ready to prove the main decomposability results of the section.

\begin{theorem}
\label{Th3} If a closed and bounded set $\mathcal{R} \subset BV_p(R^d)$ with $\mathcal{R}(0)=x_0$ is $V_p$-decomposable, then there exists a measurable and $p$-integrably bounded set-valued function $\Phi :[0,T]\rightarrow Comp(R^d)$ such that the set-valued function $t\rightarrow \mathcal{R}(t)$ is Hukuhara differentiable for almost every $t\in [0,T]$ and $D_H\mathcal{R}(t)= \overline{co}\Phi(t)$.
\end{theorem}

\begin{proof} {\rm Assume that a closed and bounded set $\mathcal{R}$ in $BV_p(R^d)$ is $V_p$-decompo-sable. It is also closed with respect to $\|\cdot \|_{\infty}$. Therefore, it follows by Theorem \ref{Th2} that $\mathcal{R}$ is an integral, i.e., there exists a measurable and a $p$-integrably bounded set-valued function $\Phi :[0,T]\rightarrow Comp(R^d)$ such that 
\[\mathcal{R} = x_0+\int \Phi(s)ds= \{f\in BV_p(R^d): f(\cdot )=x_0+\int _0^{\cdot }\phi(s)ds,\;\phi\in S_{L^p}(\Phi)\}.\]
Since $\mathcal{R}(0)=x_0$, then $\mathcal{R}(t)$ is an Aumann integral, $\mathcal{R}(t)=x_0+\int_0^t \Phi(s)ds= \{ f(t)=x_0+\int_0^t \phi(s)ds,\;\phi\in S_{L^p}(\Phi)\}$. From this we deduce that the Hukuhara derivative  $D_H(\mathcal{R}(t))$ exists for almost every $t\in [0,T]$ and $D_H\mathcal{R}(t)= \overline{co}\Phi(t)$, see e.g., \cite{Tolst}.}
\end{proof}

\begin{remark}
\label{Rem2}  {\rm If a set $\mathcal{R} \subset BV_p(R^d)$ is an integral, then  $\mathcal{R}(t)=x_0+\int_0^t \Phi(s)ds$ for every $t\in[0,T]$ and some measurable and $p$-integrably bounded set-valued function $\Phi$. The reverse implication need not hold as the following example shows.}
\end{remark}

\begin{example} {\rm Let $\Phi:\;[0,1]\rightarrow R^1$ be a constant set-valued function $\Phi (t)\equiv[0,1]$. Let $\mathcal{R}(t)=\int _0^t\Phi(s)ds=[0,t]$. Then $\mathcal{R}(\cdot )$ is Hukuhara differentiable with $D_H(\mathcal{R(\cdot )})(t)=\Phi(t)$. We will show that $\mathcal{R}=S_{V_p}(\mathcal{R}(\cdot))$ is not an integral. Let us take $f_1(t)\equiv 0$ and 

\[\begin{array}{l} \;f_2(t)= \left\{ \begin{array}{cl} t &{\rm ~ for}
\;\;\;0\leq t< 1/2 \\
-t+1 & {\rm ~ for}\;\;\;1/2\leq t\leq 1 \\
\end{array}
\right. 
\end{array}.
\]
Of course, $f_1, f_2\in S_{V_p}(\mathcal{R}(\cdot))$. However,
\[\begin{array}{l} \;f(t)=(f_1\oplus_{1/2} f_2)(t)= \left\{ \begin{array}{cl} 0 &{\rm ~ for}
\;\;\;0\leq t< 1/2 \\
-t+1/2 & {\rm ~ for}\;\;\;1/2\leq t\leq 1 \\
\end{array}
\right.     
\end{array}.
\]
Then $f(t)\notin \mathcal{R}(t)$ for $t\in [1/2,1]$, and therefore, $f=(f_1\oplus_{1/2} f_2)\!\notin \!S_{V_p}(\mathcal{R}(\cdot))=\mathcal{R}$.  It means that $\mathcal{R}$ is not an integral.}\label{Ex2}
\end{example}

\begin{theorem}
\label{Th4}  Let $F:[0,T]\rightarrow Conv(R^d)$ be a Hukuhara differentiable set-valued function, $F\in BV_p(Conv(R^d))$, $F(0)=x_0$. Then the set 
\begin{equation}
\mathcal{IS}(F) = \{f\in BV_p(R^d):\;f\in S_{V_p}(F) \hbox{ and }\;f'\in S_{L^p}(D_H(F))\} \label{IS}
\end{equation} 
is $V_p$-decomposable and therefore, it is an integral.
\end{theorem}
\begin{proof} {\rm Really, let $ f,g\in \mathcal{R}$. Then $f,g\in S_{V_p}(F)$. Therefore, $f',g'\in L^p([0,T])$ and $f',g'\in S_{L^p}(D_H(F))$. Then the function
$\gamma ={\rm 1}\hspace{-1mm}{\rm I}_{[0,a)}\cdot f'+{\rm 1}\hspace{-1mm}{\rm I}_{[a,T]}\cdot g'\in S_{L^p}(D_H(F))$ because the set $S_{L^p}(D_H(F))$ is $L^p$-decomposable. From this, we get $(f\oplus_ag)(t)=x_0+\int_0^t\gamma(s)ds\in x_0+\int_0^tD_H(F)(s)ds=F(t).$ Since $V_p((f\oplus_ag))\leq V_p(f)+V_p(g)<\infty$, then $(f\oplus_ag)\in S_{V_p}(F)$, and therefore, $(f\oplus_ag)\in \mathcal{R}$. We proved that $\mathcal{R}$ is $V_p$-decomposable and it is an integral by Theorem \ref{Th2}.}
\end{proof}

Let $C\in Conv\left( R^{d}\right) $ and let $\sigma \left( \cdot ,C\right) :R^{d}\rightarrow R^{1}$, $\sigma \left( p,C\right) =\sup_{y\in C}<p,y>$ be a support function of $C$. Let $\Sigma $ denote the unit sphere in $R^{d}$ and let $V$ denote a Lebesgue measure of a closed unit ball $B(0,1)$ in $R^{d}$, i.e., $V=\pi ^{d/2}/{\Gamma (1+\frac{d}{2})}$ with $\Gamma $ being the Euler function. Let $p_{V}$ be a normalized Lebesgue measure on $B(0,1)$, i.e., $dp_{V}=dp/V$. Let 
$$
\mathcal{M}=\{\;\mu :\mu \hbox{ is a probability measure on }B(0,1) \hbox{ having }   \label{mi} 
$$
$$
\hbox{the }C^1\!-\!\hbox{density } d\mu /dp_{V} \hbox{ with respect to measure } p_V \}.
$$
Let $\xi _{\mu }:=d\mu /dp_{V}$ and let $\nabla \xi _{\mu }$ denote the gradient of $\xi _{\mu }$. By $\omega $ we denote a Lebesgue measure on $\Sigma $. The function $St_{\mu }:Conv\left( R^{d}\right) \rightarrow R^{d}$ called a generalized Steiner center, and given by the formula
\begin{equation}
St_{\mu }(C)=V^{-1}\left( \int_{\Sigma }p\sigma \left( p,C\right) \xi _{\mu }(p)d\omega (p)-\int_{B(0,1)}\sigma \left( p,C\right) \bigtriangledown \xi _{\mu }(p)dp\right)   \label{GenSt}
\end{equation}
for every $\mu \in \mathcal{M}$,
has the following properties. 

For $A,B,C\in Conv\left( R^{d}\right)$ and $a,b\in R^{1}$
$$
St_{\mu }(C)\in C\hbox{, }St_{\mu }(aA+ bB)=aSt_{\mu }(A)+bSt_{\mu }(B), \label{NirGenSt1}
$$
\begin{equation}
\|St_{\mu }(A)-St_{\mu }(B)\|\leq  L_{\mu }\cdot H_{R^{d}}\left( A,B\right) ,
\label{NierGenSt}
\end{equation}
where $L_{\mu }= d\max_{p\in \Sigma }\xi _{\mu }(p)+\max_{p\in B(0,1)}\|\bigtriangledown \xi _{\mu }(p)\|$ (see e.g., \cite{Dentch2}).

Since the set $C^1_d=\{\xi \in C^1(B(0,1),R^+):\;\int_B\xi dp_V=1\}$ is separable then there exists a countable subset $\{\xi _n\}\subset C^1_d$ dense  in $C^1_d$ with respect to supremum norm. Let $\{\mu _n\}$ be a sequence of measures from $\mathcal{M}$ with densities $\{\xi _n\}$. It is known that every set $C\subset Conv(R^d)$ has a representation 
$$
C=\overline{\{ St_{\mu}(C)\}_{\mu \in \mathcal{M}}},
$$
where $St_{\mu}(C)$ are generalized Steiner points of $C$ given by formula (\ref{GenSt}), see also \cite{Dentch2}. Therefore, by separability of $C^1_d$, we have
$$
C=\overline{\{St_{{\mu}_n}(C)\}_{n=1}^{\infty}}.
$$
Let $F:[0,T]\rightarrow Conv(R^d)$ be a set-valued function. Then 
\begin{equation}\label{intSt}
St_{\mu}\left(\int_0^tF(s)ds\right)=\int_0^t \left( St_{{\mu}}(F(s)\right)ds
\end{equation}
for every $t\in [0,T]$ by \cite{BaiFar} and we obtain
\begin{equation}\label{Cast}
\int_0^tF(s)ds=\overline{\{ St_{\mu _n}(\int_0^tF(s)ds)\}_{n=1}^{\infty}}=\overline{ \{\int_0^t \left( St_{{\mu _n}}(F(s)\right)ds\}_{n=1}^{\infty}}.
\end{equation}
Assume that $F\in BV_p(Conv(R^d))$ is Hukuhara differentiable, $F(0)=x_0$, and consider again a set $\mathcal{IS}(F)$ defined by  $(\ref{IS})$. This set is an integral by Theorem \ref{Th4}. 
We prove the following result.
\begin{theorem} \label{Th5}
Let $F\in BV_p(Conv(R^d))$ be a Hukuhara differentiable set-valued function, $F(0)=x_0$. Then there exists a Castaing representation $\{f_n\}_{n=1}^{\infty}$ of $F$ with $f_n\in \mathcal{IS}(F)$ for every $n=1,2,...$ .
\end{theorem}
\begin{proof} {\rm Since $F(t)=x_0+\int _0^t D_H(F)(s)ds$ then by formula (\ref{Cast}) we obtain
$$
F(t)=x_0+\int _0^t D_H(F)(s)ds=x_0+\overline{ \{\int_0^t  St_{\mu _n}(D_H(F)(s))ds\}_{n=1}^{\infty}}.
$$
It means that the sequence $\{f_n\}_{n=1}^{\infty}$ defined by the formula $f_n(t)=x_0+\int_0^t St_{{\mu _n}}(D_H(F)(s))ds$ is a Castaing representation of $F$. Moreover, $f_n'(t)= St_{{\mu _n}}(D_H(F)(t))\in D_H(F)(t)$. We have to show that $f_n\in BV_p(R^d)$ and $f_n'\in L^p([0,T])$. We know that $f_n(t)=x_0+\int_0^t St_{{\mu _n}}(D_H(F)(s))ds=x_0+St_{\mu _n}\left(F(t)\right)$
by equality (\ref{intSt}). It follows from formula (\ref{NierGenSt}) that
$$
\|f_n(t)-f_n(s)\|\leq  L_{\mu }\cdot H_{R^{d}}\left( F(t),F(s)\right),
$$
where $L_{\mu }= d\max_{p\in \Sigma }\xi _{\mu }(p)+\max_{p\in B(0,1)}\|\bigtriangledown \xi _{\mu }(p)\|$. Therefore, for every $0\leq a<b<\leq T$,
$$
V_p(f_n,[a,b])= \sup _{\Pi _m} \sum _{i=0}^{m}\frac{\|f_n(t_i)-f_n(t_{i-1})\|^p}{(t_{i}-t_{i-1})^{p-1}}
$$
$$
\leq L_{\mu}\; \sup _{\Pi _m} \sum _{i=0}^{m}\frac{\big(H_{R^d}(F(t_{i-1}),F(t_i))\big)^p}{(t_{i}-t_{i-1})^{p-1}}=V_p(F,[a,b])<\infty .
$$
Therefore, $f_n\in BV_p(R^d)$. 
\bigskip

Now, we are able to apply Corollary 3.4(a) from \cite{Chi} to deduce that $f_n'$ satisfies $\int _0^t \|f_n'(s)\|^pds<\infty $. Since $f_n'$ is a measurable selection of $D_H(F)$ then $f_n'\in S_{L^p}(D_H(F))$.  Therefore, $f_n\in \mathcal{IS}(F)$ for every $n=1,2,...$ .}
\end{proof}

\section{Set-valued Young integrals}

At the beginning of this section we recall the notion of a Young integral in a single valued case introduced by L.S. Young in \cite{Young}. For details see also \cite{Friz}. Let $f:[0,T]\rightarrow R^{d}$ and $g:[0,T]\rightarrow R^{d}$ be given functions. For the partition $\Pi _m :0=t_{0}<t_{1}<...<t_{m}=T$ of the interval $[0,T]$ we consider the Riemann sum of $f$ with respect to $g$
$$
S(f,g,\Pi _m ):=\sum\limits_{i=1}^{m}f\left( t_{i-1}\right) (g(t_{i})-g\left(t_{i-1}\right) )\hbox{.}
$$
Let $|\Pi _m |:=\max \{t_{i}-t_{i-1}:1\leq i\leq m-1\}$. Then the following version of Proposition 2.4 in \cite{FrizZhang} holds.
\begin{proposition}
\bigskip \label{prop4} Let $f\in BVar_{p}(R^{d})$ and $g\in \mathcal{C}^{\alpha }\left( R^{1}\right) $ where $1/p+\alpha >1$. Then the limit 
$$
\lim_{|\Pi _m |\rightarrow 0}S(f,g,\Pi _m )=:\int_{0}^{T}fdg  \label{Riemannlim}
$$
exists and the inequality 
\begin{equation}
\|\int_{s}^{t}fdg-f(s)(g(t)-g(s))\|\leq C(\alpha ,p)\left(Var_{p}(f)\right) ^{1/p}M_{\alpha }\left( g\right) (t-s)^{\alpha }
\label{pnier}
\end{equation}
holds for every $0\leq s<t\leq T$, where the constant $C(\alpha ,p)$ depends only on $p$ and $\alpha $.
\end{proposition}
\begin{corollary}
\label{Cor1}\bigskip Let $f_{1}$, $f_{2}\in BVar_{p}(R^{d})$ and $g\in \mathcal{C}^{\alpha }\left( R^{1}\right) $ where $1/p+\alpha >1$. Then 
$$
\left\Vert \int_{0}^{\cdot }f_{1}dg-\int_{0}^{\cdot }f_{2}dg\right\Vert _{\alpha }$$
$$
\leq \left( \|f_{1}-f_{2}\|_{\infty }+C(\alpha
,p)\left( Var_{p}(f_{1}-f_{2})\right) ^{1/p}\right) M_{\alpha}\left( g\right) (1+T^{\alpha })\hbox{. }
$$
\end{corollary}

\bigskip

In the case $f\in \mathcal{C}^{\beta }(R^{d})$ and $\alpha ,\beta \in (0,1]$ with $\alpha +\beta >1$, one can express the Young integral by fractional derivatives. Namely, let
\[f_{0+}(t)=\left(f(t)-f(0+)\right) I_{\left( 0,T\right) }\left( t\right) \hbox{ and }f_{T-}(t)=\left( f(t)-f(T-)\right) I_{\left( 0,T\right) }\left( t\right). \]
The right-sided and left-sided fractional derivatives of order $0<\rho <1$ for the function $f:[0,T]\rightarrow R^{1}$ are defined by
$$
D_{0+}^{\rho }f(t)=\frac{1}{\Gamma \left( 1-\rho \right) }\left( \frac{f(t)}{t^{\rho }}+\rho \int_{0}^{t}\frac{f(t)-f(s)}{\left( t-s\right) ^{\rho +1}}ds\right)   \label{Frderiv1}
$$
and 
$$
D_{T-}^{\rho }f(t)=\frac{\left( -1\right) ^{\rho }}{\Gamma \left( 1-\rho \right) }\left( \frac{f(t)}{\left( T-t\right) ^{\rho }}+\rho \int_{t}^{T}\frac{f(t)-f(s)}{\left( s-t\right) ^{\rho +1}}ds\right) \hbox{. }
\label{Frderiv2}
$$
\bigskip

Then the following result holds, see e.g., \cite{Samko}.

\begin{proposition}
\label{Prop5}\bigskip Suppose that $g:[0,T]\rightarrow R^{1}$, $g\in \mathcal{C}^{\alpha }\left( R^{1}\right) $ and $f\in \mathcal{C}^{\beta }\left( R^{d}\right) $. Then the integral $\int_{0}^{T}fdg$ exists in the sense of Riemann and 
$$
\int_{0}^{T}fdg=\left( -1\right) ^{\rho }\int_{0}^{T}D_{0+}^{\rho
}f_{0+}(t)D_{T-}^{1-\rho }g_{T-}(t)dt+f(0)(g(T)-g(0))\hbox{ }  \label{RSint}
$$
for every $\rho \in (1-\alpha ,\beta )$. Moreover, the following version of the inequality {\rm (\ref{pnier})} 
$$
\|\int_{t_{1}}^{t_{2}}fdg-f(t_{1})(g(t_{2})-g(t_{1}))\|\leq C(\alpha ,\beta )M_{\alpha }\left( g\right) M_{\beta }(f)(t_{2}-t_{1})^{\alpha +\beta }  \label{Cnier}
$$
holds for every $0\leq t_{1}<t_{2}\leq T$, where $C(\alpha ,\beta )$ depends only on $\alpha $ and $\beta $.
\end{proposition}

\bigskip

Let us consider again a set $\mathcal{IS}(F)$ given in (\ref{IS})
$$
\mathcal{IS}(F) = \{f\in BV_p(R^d):\;f\in S_{V_p}(F) \hbox{ and }\;f'\in S_{L^p}(D_H(F))\}. 
$$ 
 
\begin{definition} {\rm We define {\it a set-valued Young integral} of Hukuhara differentiable $F\in BV_p(ConvR^n)$ with respect to a function $g\in \mathcal{C}^{\alpha }\left( R^{1}\right) $, $1/p+\alpha >1$, by the formula
$$
(\mathcal{IS})\int_{0}^{t}Fdg:=cl_{R^{d}}\left\{\int_{0}^{t}fdg:f\in \mathcal{IS}(F)\right\} \hbox{, }  \label{fpint}
$$
for every $1/p+\alpha>1$ and $g\in C^{\alpha}(R^1)$.}

\end{definition}

\bigskip

We have
$$
\| (\mathcal{IS})\int_{s}^{t}Fdg\|\leq M_{\alpha }\left( g\right) \left\Vert F\right\Vert _{\beta }\left( 1+C(\alpha ,\beta )T^{\beta }\right) (t-s)^{\alpha }\hbox{ }
\label{efpeineq}
$$
for $0\leq s\leq t\leq T$. Since $F$ and $D_H(F)$ take on convex values then the sets $SV_p(F)$ and $S_{L^p}(D_H(F))$ are convex and therefore, $\mathcal{IS}(F)$ and $(\mathcal{IS})\int_{0}^{t}Fdg$ for every $t\in [0,T]$ are convex subsets of $BV_p(R^d)$ and $R^d$, respectively.
 
\bigskip

The folllowing lemma was proved in \cite{MarJur1}.

\begin{lemma}
\label{Lemma1}Let $g\in \mathcal{C}^{\alpha }\left( R^{1}\right) $. Then, for every $\rho \in (1-\alpha ,\beta )$, there exists a positive constant $C(\rho )$ such that for every $f_{1},f_{2}\in \mathcal{C}^{\beta }\left( R^{d}\right) $, $t\in \lbrack 0,T]$ and $\theta \in (0,1]$ the inequality
$$
\| \int_{0}^{t}f_{1}dg-\int_{0}^{t}f_{2}dg\|\leq C(\rho )\left[ \|f_{1}-f_{2}\|_{\infty }+(M_{\beta }\left( f_{1}\right) +M_{\beta }\left( f_{2}\right) )\theta ^{\beta }\right] \theta ^{-\rho }
$$
$$
+\|f_{1}(0)-f_{2}(0)\||g(T)-g(0)|\hbox{.}\label{intest}   
$$
holds.
\end{lemma}   

Using this Lemma we are able to prove the following result. 

\begin{theorem}
\label{Th6} For every $\rho \in (1-\alpha ,\beta )$ there exists a positive constant $C(\rho )$ such that for every $\theta \in (0,1]$, $t\in \lbrack 0,T]$ and for every Hukuhara differentiable set-valued functions $F_1,\;F_2$ with bounded Hukuhara derivatives, the inequality
$$
H_{R^{d}}\left( (\mathcal{IS})\int_{0}^{t}F_{1}dg,(\mathcal{IS})\int_{0}^{t}F_2dg\right)
$$
$$
 \leq C(\rho )\left( \int_{0}^{T}H_{R^{d}}(D_H(F_1)(s),D_H(F_2)(s))ds\right. 
$$
\begin{equation}
\left. +(T+T^{1-\beta })(\sup_{t\in \lbrack 0,T]}\| D_H(F_1)(t)\| +\sup_{t\in \lbrack 0,T]}\| D_H(F_2)(t)\|)\theta ^{\beta }\right) \theta ^{-\rho }\label{fpinintineq}
\end{equation}
$$
+M_{\alpha }\left( g\right) T^{\alpha }\int_{0}^{T}H_{R^{d}}\left(D_H(F_1)(s),D_H(F_2)(s)\right)ds.  
$$
holds.
\end{theorem}

\begin{proof} {\rm Let $F:[0,T]\rightarrow Conv\left( R^{d}\right) $ be Hukuhara differentiable. If the set-valued function $D_H(F)(\cdot )$ is bounded, i.e., $\sup_{t\in [0,T]}\| D_H(F)(t)\|<\infty $, then $\| F\| _{\infty }\leq T\sup_{t\in [0,T]}\| D_H(F)(t)\| $ and $V_p(F)\leq T\sup_{t\in [0,T]}\|D_H(F)(t)\|^p$ as well as $M_{\beta }(F)\leq T^{1-\beta }\sup_{t\in [0,T]}\| D_H(F)(t)\| $ for any $p\geq 1$ and $\beta \in (0,1)$. Thus $F\in BV_p( Conv( R^{d}) ) $ and $S_{V_p }(F)\neq \emptyset $.  

{\rm  We obtain by Lemma \ref{Lemma1}, for any $f_{1}\in  \mathcal{IS}(F_1)$, $f_{2}\in ,\mathcal{IS}(F_2)$, $\rho \in (1-\alpha ,\beta )$, $t\in \lbrack 0,T]$ and $\theta \in (0,1]$
$$
\| \int_{0}^{t}f_{1}dg-\int_{0}^{t}f_{2}dg\|
$$
$$
\leq C(\rho )\left[ \|f_{1}-f_{2}\|_{\infty }+(M_{\beta }\left( f_{1}\right) +M_{\beta }\left( f_{2}\right) )\theta ^{\beta }\right] \theta ^{-\rho }+\|f_{1}-f_{2}\|_{\infty }M_{\alpha }(g)T^{\alpha }.
$$
 Thus 
$$
\mathrm{dist}_{R^{d}}\left( \int_{0}^{t}f_{1}dg, (\mathcal{IS})\int_{0}^{t}F_{2}dg\right)  
$$
$$
\leq C(\rho )\!\left[ \mathrm{dist}_{\infty }\left( f_{1},\mathcal{IS}(F_2)\right) \!+\!(\!\sup_{f_{1}\in \mathcal{IS}(F_1)}\!M_{\beta }\left( f_{1}\right)\! +\!\sup_{f_{2}\in \mathcal{IS}(F_2)}\!M_{\beta }\left( f_{2}\right) )\theta ^{\beta }\right] \theta^{-\rho } 
$$
$$
+\mathrm{dist}_{\infty }\left( f_{1},\mathcal{IS}(F_2)\right) M_{\alpha }(g)T^{\alpha }\hbox{.}
$$
Hence, 
$$
\overline{H}_{R^{d}}\left((\mathcal{IS})\int_{0}^{t}F_{1}dg,(\mathcal{IS})\int_{0}^{t}F_{2}dg\right) \hbox{ }
$$
$$
 \leq C(\rho )\!\left[ H_{\infty }\left( \mathcal{IS}(F_1),\mathcal{IS}(F_2)\right)\!+\!(\!\sup_{f_{1}\in \mathcal{IS}(F_1)}\!M_{\beta }\left(f_{1}\right) \!+\!\!\sup_{f_{2}\in \mathcal{IS}(F_2)}\!M_{\beta }\left(f_{2}\right) )\theta ^{\beta }\right] \theta ^{-\rho }
$$
$$
+M_{\alpha }\left( g\right) T^{\alpha }H_{\infty }\left( \mathcal{IS}(F_1),\mathcal{IS}(F_2)\right) . 
$$

The same estimation holds for $\overline{H}_{R^{d}}\left((\mathcal{IS})\int_{0}^{t}F_{1}dg,(\mathcal{IS})\int_{0}^{t}F_{2}dg\right)$. 

Therefore,
\begin{equation}
H_{R^{d}}\left((\mathcal{IS})\int_{0}^{t}F_{1}dg,(\mathcal{IS})\int_{0}^{t}F_{2}dg\right) \label{fpinintineq1}
\end{equation}
$$
 \leq C(\rho )\!\left[ H_{\infty }\left( \mathcal{IS}(F_1),\mathcal{IS}(F_2)\right)\!+\!(\!\sup_{f_{1}\in \mathcal{IS}(F_1)}\!M_{\beta }\left(f_{1}\right) \!+\!\!\sup_{f_{2}\in \mathcal{IS}(F_2)}\!M_{\beta }\left(f_{2}\right) )\theta ^{\beta }\right] \theta ^{-\rho }
$$
$$
+M_{\alpha }\left( g\right) T^{\alpha }H_{\infty }\left( \mathcal{IS}(F_1),\mathcal{IS}(F_2)\right) . 
$$
}
Suppose that $f\in \mathcal{IS}(F)$. Then, it is expressed as the integral, i.e., $f(\cdot )=\int_{0}^{\cdot }\phi (s)ds$ for some $\phi \in S_{L^p}(D_H(F))$ and we have
\[M_{\beta }(f)=\sup_{0\leq s<t\leq T}\frac{\|f(t)-f(s)\|}{\left( t-s\right) ^{\beta }}=\sup_{0\leq s<t\leq T}\frac{\|\int_s^t\phi(\tau)d\tau\|}{\left( t-s\right) ^{\beta }}.\]
From the other side, we get by the formula (\ref{33}), the equalities
\[M_{\beta }(F)=\sup_{0\leq s<t\leq T}\frac{H_{R^d}\left( F(t),F(s)\right) }{\left( t-s\right)^{\beta }}\]
\[=\sup_{0\leq s<t\leq T}\frac{H_{R^d}\left(\int _s^t D_HF(\tau)d\tau +\int _0^s D_HF(\tau)d\tau,0+\int _0^s D_HF(\tau)d\tau\right) }{\left( t-s\right)^{\beta }}\]
\[=\sup_{0\leq s<t\leq T}\frac{H_{R^d}\left(\int _s^t D_HF(\tau)d\tau ,0 \right)}{\left( t-s\right)^{\beta }}=\sup_{0\leq s<t\leq T}\frac{\|\int _s^t D_HF(\tau)d\tau \| }{\left( t-s\right)^{\beta }}\]
Therefore, 
\[M_{\beta }(f)\leq M_{\beta }(F)\]
 
and taking in the mind the begining of the proof, we get
\begin{equation}
\sup_{f\in \mathcal{IS}(F)}\|f\| _{\beta }\leq \| F\| _{\beta }\leq ( T+T^{1-\beta }) \sup_{t\in [0,T]}\| D_H(F)(t)\| <\infty.
\label{selineq}
\end{equation}
}
{\rm Let $\phi _1\in S_{L^p}(D_H(F_1))$. Then, by Theorem 2.2 from \cite{Hiai}, we have
$$
\inf_{\phi _2\in S_{L^p}(D_H(F_2))}\sup_{t\in \lbrack 0,T]}\| \int_{0}^{t}\phi _1(s)ds-\int_{0}^{t}\phi _2(s)ds\|
$$
$$
\leq \inf_{\phi _2\in S_{L^p}(D_H(F_2))}\int_{0}^{T}\| \phi _1(s)-\phi _2(s)\|ds=\int_{0}^{T}\mathrm{dist}_{R^{d}}(\phi _1(s),D_H(F_2)(s))ds
$$
$$
\leq \int_{0}^{T}H_{R^{d}}(D_H(F_1)(s),D_H(F_2)(s))ds.
$$
Thus 
$$
\overline{H}_{\infty }\left( \mathcal{IS}(F_1),\mathcal{IS}(F_2)\right) \leq \int_{0}^{T}H_{R^{d}}(D_H(F_1)(s),D_H(F_2)(s))ds.
$$
In a similar way we get 
$$
\overline{H}_{\infty }\left( \mathcal{IS}(F_2),\mathcal{IS}(F_1)\right) \leq \int_{0}^{T}H_{R^{d}}(D_H(F_1)(s),D_H(F_2)(s))ds
$$
and finally
$$
H_{\infty }\left( \mathcal{IS}(F_1),\mathcal{IS}(F_2)\right)\hbox{ }\leq \int_{0}^{T}H_{R^{d}}(D_H(F_1)(s),D_H(F_2)(s))ds.
$$
Hence, by  the formula (\ref{fpinintineq1}) together with (\ref{selineq}), we obtain inequality (\ref{fpinintineq}).}
\end{proof}

\begin{corollary}
\label{Cor22}Let $F_{n},F$ be Hukuhara differentiable set-valued functions with bounded Hukuhara derivatives satisfying
$$
\sup_{t\in \lbrack 0,T]}H_{R^{d}}(D_H(F_n)(t),D_H(F)(t))\rightarrow 0\hbox{ as }n\rightarrow \infty .
$$
Then, 
\begin{equation}
\sup_{t\in [0,T]}H_{R^{d}}\left( (\mathcal{IS})\int_{0}^{t}F_ndg,(\mathcal{IS})\int_{0}^{t}Fdg\right) \rightarrow 0\hbox{ as }n\rightarrow \infty \hbox{.}  \label{intlim}
\end{equation}
\end{corollary}

\begin{proof} {\rm Let us note that for every $n\geq 1$ we have
$$
\left\vert \sup_{t\in \lbrack 0,T]}\| D_H(F_n)(t)\|-\sup_{t\in \lbrack 0,T]}\| D_H(F)(t)\|\right\vert 
$$
$$
\leq \sup_{t\in \lbrack 0,T]}H_{R^{d}}(D_H(F_n)(t),D_H(F)(t)).
$$
Thus, 
$$
\sup_{t\in \lbrack 0,T]}\| D_H(F_n)(t)\|\rightarrow \sup_{t\in \lbrack 0,T]}\| D_H(F)(t)\|\hbox{ as }n\rightarrow \infty .
$$
Therefore, the sequence $\left( \sup_{t\in \lbrack 0,T]}\| D_H(F_n)(t)\|  \right) _{n\geq 1}$ is bounded. Hence we get
by Theorem \ref{Th6} 
$$
\limsup_{n}\left( \sup_{t\in \lbrack 0,T]}H_{R^{d}}\left( (\mathcal{IS})\int_{0}^{t}F_{n}dg,(\mathcal{IS})\int_{0}^{t}Fdg\right) \right) 
$$
$$
\leq C(\rho )(T+T^{1-\beta })\left( \sup_{n}\sup_{t\in \lbrack 0,T]}\| D_H(F_n)(t)\|+\sup_{t\in \lbrack 0,T]}\| D_H(F)(t)\|\right) \theta ^{\beta -\rho }\hbox{.}
$$
Since $\beta >\rho $ and $\theta \in (0,1]$ is arbitrarily taken, we obtain formula (\ref{intlim})}.
\end{proof}
\bigskip

It was proved in \cite{Hiai} that for measurable and $p$-integrably bounded set-valued functions $F_1,F_2:[0,T]\rightarrow Comp(R^d)$, the equality 
\begin{equation}
S_{L^p}\left(cl_{R^d}\{F_1+F_2)(\cdot )\}\right)=cl_{L^p}\{S_{L^p}(F_1)+S_{L^p}(F_2)\} 
\label{L^psel}
\end{equation}
holds. Therefore, a set-valued Aumann integral satisfies 
\[\int_0^t(F_1+F_2)dt=\int_0^tF_1dt+\int_0^tF_2dt.\]
 We will show that a set-valued Young integral is additive also.

\begin{theorem} \label{Th7} Let $F,F_1,F_2\in BV_p(Conv(R^d))$ be Hukuhara differentiable with $p$-integrably bounded Hukuhara derivatives, $1<p<\infty$. Let $g\in C^{\alpha}(R^1)$, where $1/p+\alpha>1.$ Then 
\[(\mathcal{IS})\int_{0}^{t}(F_1+F_2)dg=(\mathcal{IS})\int_{0}^{t}F_1dg+(\mathcal{IS})\int_{0}^{t}F_2dg.\] 
Moreover, if the set $F(0)$ is bounded in $R^d$, then $(\mathcal{IS})\int_0^{\cdot}Fdg$ and $(\mathcal{IS})\int_s^{t}Fdg$ are bounded sets in $C^{\alpha}(R^d)$ and in $R^d$, respectively.
\end{theorem} 
 
\begin{proof} {\rm We show that $\mathcal{IS}(F_1+F_2)=\mathcal{IS}(F_1)+\mathcal{IS}(F_2)$. Let us take an arbitrary $f\in \mathcal{IS}(F_1+F_2)$. Then $f\in SV_p(F_1+F_2)$, $f'\in S_{L^p}(D_H(F_1+F_2))$ and $f(\cdot )=\int_0^{\cdot}f'(s)ds.$ Since $D_H(F_i)(t)$ takes on compact and convex values in $R^d$, then $f'\in S_{L^p}(D_H(F_1+F_2))=cl_{L^p}\{S_{L^p}(D_H(F_1))+S_{L^p}(D_H(F_2))\}$ by equality (\ref{L^psel}). Therefore, there exist sequences $(\phi_n^1)\subset S_{L^p}(D_H(F_1))$ and $(\phi_n^2)\subset S_{L^p}(D_H(F_2))$ such that $\phi_n^1+\phi_n^2\rightarrow f'$ with respect to the $L^p$-norm convergence. But $S_{L^p}(D_H(F_1))$ is a closed, convex and bounded subset of $L^p$ by \cite{Hiai}, and therefore, weakly compact.  Then there exists a subsequence $(\phi_{n_k}^1)$  weakly convergent to some $\phi ^1\in S_{L^p}(D_H(F_1))$. Similarly,  passing to the subsequence if needed, $(\phi_{n_k}^2)$ tends weakly to some  $\phi ^2\in S_{L^p}(D_H(F_2))$. Therefore,  $f'=\phi^1+\phi^2\in S_{L^p}(D_H(F_1))+S_{L^p}(D_H(F_2))$. But $f(\cdot )=\int _0^{\cdot}f'(s)ds=\int _0^{\cdot}\phi^1(s)ds+\int _0^{\cdot}\phi^2(s)ds.$ Since  $\int _0^{t}\phi^1(s)ds\in \int _0^{t}D_H(F_1)(s)ds=F(t)$, then $\int _0^{\cdot}\phi^1(s)ds\in \mathcal{IS}(F_1)$. In the same way,  $\int _0^{\cdot}\phi^2(s)ds\in \mathcal{IS}(F_2)$, and therefore, $\mathcal{IS}(F_1+F_2)\subset \mathcal{IS}(F_1)+\mathcal{IS}(F_2)$. For the proof of a reverse inclusion it is enough to note that taking $f_1\in \mathcal{IS}(F_1)$ and $f_2\in\mathcal{IS}(F_2)$ their sum belongs to $SV_p(F_1+F_2)$ and the sum of their derivatives belongs to $S_{L^p}(D_H(F_1+F_2))$. Hence, $f_1+f_2\in \mathcal{IS}(F_1+F_2)$. 

Now let us remark that the operator $J:BV_p(R^d)\rightarrow R^d$ defined by the formula $J(f)=\int_0^tfdg$ is linear. From this we get
\[(\mathcal{IS})\int_0^t(F_1+F_2)dg=J(\mathcal{IS}(F_1+F_2))=J(\mathcal{IS}(F_1)+\mathcal{IS}(F_2))\]
\[=J(\mathcal{IS}(F_1))+J(\mathcal{IS}(F_2))=(\mathcal{IS})\int_0^t(F_1)dg+(\mathcal{IS})\int_0^t(F_2)dg.\]
 Hence the first statement follows.

\bigskip

 We show that the set $\mathcal{IS}(F)$ is bounded in the space $(BV_p(R^d),\|\cdot \|_{V_p})$.  Let $f\in \mathcal{IS}(F)$ be arbitrarily taken. By assumption, the set $S_{L^p}(D_H(F))$ is  bounded in $L^p$ norm by some constant $M$. Since $f'\in S_{L^p}(D_H(F))$ then $(V_p(f))^{1/p}=(\int _0^T\|f'(s)\|^pds)^{1/p}\leq M$. 

Moreover, $\|f(t)-f(0)\|\leq T^{1-1/p}(V_p(f))^{1/p}$ for every $t\in [0,T]$ by Proposition \ref{Prop1}(b). This implies
$$
\|f\|_{V_p}=\|f\|_{\infty}+(V_p(f))^{1/p}\leq \sup_{h\in \mathcal{IS}(F)}\|h(0)\|+T^{1-1/p}M+M
$$
Then, for every $f\in \mathcal{IS}(F)$, we get by Corollary \ref{Cor1}
$$
\| \int_{0}^{\cdot }fdg\| _{\alpha }\leq \left( \|f\|_{\infty }+C(\alpha ,p)\left( Var_{p}(f)\right) ^{1/p}\right) M_{\alpha}\left( g\right) (1+T^{\alpha })
$$
$$
\leq \sup_{h\in \mathcal{IS}(F)}\|h(0)\|+(T^{1-1/p})M+C(\alpha ,p)MT^{1-1/p} M_{\alpha}\left( g\right) (1+T^{\alpha })
$$
$$
=\sup_{h\in \mathcal{IS}(F)}\|h(0)\|+(T^{1-1/p})M(1+C(\alpha ,p) M_{\alpha}\left( g\right) (1+T^{\alpha })).
$$
Using Corollary \ref{Cor1} once again, we obtain in a similar way
$$
\| \int_{s}^{t}fdg\|\leq \left( \|f\|_{\infty }+C(\alpha ,p)\left( Var_{p}(f)\right) ^{1/p}\right) M_{\alpha}\left( g\right) (t-s)^{\alpha }
$$
$$
\leq \sup_{h\in \mathcal{IS}(F)}\|h(0)\|+(T^{1-1/p})M(1+C(\alpha ,p) M_{\alpha}\left( g\right) T^{\alpha }).
$$
Thus we obtain the appropriate boundedness of both integrals.
}
\end{proof}

\section{Disclosure of potential conflicts of interest}

Conflict of Interest: The authors declare that they have no conflict of interest.

\section{Declarations:}

To be used for non-life science journals

Funding (information that explains whether and by whom the research was supported) Not applicable

Conflicts of interest/Competing interests (include appropriate disclosures) Not applicable

Availability of data and material (data transparency) Not applicable

Code availability (software application or custom code) Not applicable


\begin{thebibliography}{99}

\bibitem{Ahmed1} N.U. Ahmed, Semigroup theory with applications to systems and control, Longman, Pitman Research Notes in Math. Series, Harlow, Essex1991.

\bibitem{Ahmed2} N.U. Ahmed, Optimal relaxed controls for nonlinear infinite dimensional stochastic differential inclusions, Marcel Dekker, Lecture Notes in Pure and Appl. Math., \textbf{180} (1994), 1-19. 

\bibitem{Aubin} J.P. Aubin, Viability Theory, Birkh\"{a}user, Boston, 2009.

\bibitem{AubinCell} J.P. Aubin, A. Cellina, Differential Inclusions, Springer, 1984.

\bibitem{AubinFr} J.P. Aubin, H. Frankowska, Set-Valued Analysis, Birkh\"{a}user, Boston, 1990.

\bibitem{Aumann} R.J. Aumann, Integrals of set-valued functions, J. Math. Anal. Appl. \textbf{12} (1965), 1-12.

\bibitem{Bailleul} I. Bailleul, A. Brault, L. Coutin, Young and rough differential inclusions, arXiv: 1812.06727v2 [math.CA] 5 Jun 2019.

\bibitem{BaiFar}  R. Baier, E. Farkhi,  Regularity and integration of set-valued maps represented by generalized Steiner points, Set-Valued Anal. \textbf{15} (2007), 185-207.

\bibitem{Cas1} C. Castaing,  Sur l'existence des sections s\'epar\'ement mesurables et s\'epar\'ement continues d'une multi-application, S\'eminaire d'Analyse Convexe, Univ. des Sci. et Techniques du Languedoc Montpellier (1975) Expose No. \textbf{14}.

\bibitem{ChistGalkin} V.V. Chistyakov, O.E. Galkin, On maps of bounded $p$-variation with $p>1$, Positivity \textbf{2} (1998), 19-45.

\bibitem{Chi}  V. Chistyakov, Selections of bounded variation, J. Appl. Anal. \textbf{10 (1)} (2004), 1-82.

\bibitem{Coutin} L. Coutin, Z. Qian, Stochastic analysis, rough paths analysis and fractional Brownian motions, Probab. Theory Related Fields \textbf{122} (2002), 108-140.

\bibitem{Dentch2} D. Dentcheva, Differentiable selections and Castaing representations of multifunctions, J. Math. Anal. Appl. \textbf{223} (1998), 371-396.

\bibitem{DiamKloeden} P. Diamond and P. Kloeden, Metric Spaces of Fuzzy Sets: Theory and Applications, World Scientific, 1994.

\bibitem{Fei} W.Y. Fei, D.F. Xia, On solutions to stochastic set differential equations of It\^o type under the non-Lipschitzian condition, Dynam. Syst. Appl. \textbf{22 (1)} (2013), 137-156.

\bibitem{Friz} P.K. Friz, N.B. Victoir, Multidimensional stochastic processes as rough paths: theory and applications, Cambridge Studies in Advanced Mathematics, 2010. 

\bibitem{FrizZhang} P.K. Friz, H. Zhang, Differential equations driven by rough paths with jumps, J. Differential Equations \textbf{264} (2018), 6226-6301.

\bibitem{Frysz} A. Fryszkowski, Fixed Point Theory for Decomposable Sets, Kluwer Academic Publishers, Dordrecht, 2004.

\bibitem{Gorn} S. Djebali, L. G\'orniewicz, A. Ouahab, Solutions sets for differential equations and inclusions, Series in Nonlin. Anal. Appl., De Gruyter GmbH, Berlin, Boston, 2013.

\bibitem{Hiai} F. Hiai, H. Umegaki, Integrals, conditional expectations and martingales for multivalued functions, J. Multivariate Anal. \textbf{7 (1)} (1977), 149--182. 

\bibitem{Kis0} M. Kisielewicz, Differential Inclusions and Optimal Control, Kluwer Acad. Publ. Dordrecht, 1991.

\bibitem{Kis1} M. Kisielewicz, Stochastic Differential Inclusions and Applications, Springer, New York, 2013.

\bibitem{Laksh5} V. Lakshmikantham, T. Gnana Bhaskar, J. Vasundhara Devi, Theory of Set Differential Equations in a Metric Space, Cambridge Scientific Publishers, Cambridge, 2006.

\bibitem{LakshTolst} V. Lakshmikantham and A.A. Tolstonogov, Existence and interrelation between set and fuzzy differential equations, Nonlinear Anal. \textbf{55} (2003), 255-268.

\bibitem{Lelay} A. Lejay, Controlled differential equations as Young integrals: A simple approach, J. Differential Equations \textbf{249} (2010), 1777-1798.

\bibitem{Lyons} T. Lyons, Differential equations driven by rough signals, Rev. Mat. Iberoam. \textbf{14} (1998), 215-310.

\bibitem{MarJur1} M. Michta, J. Motyl, Selections properties and set-valued Young integrals of set-valued functions, Results Math. \textbf{75:164} (2020), DOI: 10.1007/s00025-020-01284-3.

\bibitem{Samko} S.G. Samko, A.A. Kilbas, O.I. Marichev, Fractional Integrals and Derivatives, Theory and Applications, Gordon and Breach Science Publishers, Yverdon, 1993.

\bibitem{Tolst} A.A. Tolstonogov, Differential Inclusions in a Banach Space, Kluwer Academic Publishers, Dordrecht, 2000.

\bibitem{Young} L.S. Young, An inequality of the H\"{o}lder type connected with Stieltjes integration, Acta Math. \textbf{67} (1936), 251-282.

\end{thebibliography}
\end{document}